\documentclass[12pt]{amsart}
\usepackage{amsthm}
\usepackage{amssymb}
\usepackage{amsmath}
\theoremstyle{plain}
\newtheorem{proposition}{Proposition}
\newtheorem{theorem}[proposition]{Theorem}
\newtheorem{lemma}[proposition]{Lemma}
\newtheorem{corollary}[proposition]{Corollary}
\newtheorem{hypothesis}[proposition]{Hypothesis}
\theoremstyle{definition}
\newtheorem{definition}[proposition]{Definition}

\theoremstyle{definition}

\newtheorem{remark}[proposition]{Remark}

\numberwithin{equation}{section}

\numberwithin{proposition}{section}

\gdef\myletter{}

\let\savetheequation\theequation
\def\theequation{\savetheequation\myletter}

\newcommand{\CC}{{\mathbb C}}

\newcommand{\RR}{{\mathbb R}}

\renewcommand{\date}{\today}

\begin{document}

\author{Thomas Bloom*}
\thanks{*supported by an NSERC of Canada grant}
\subjclass{60B20,31A15}
\keywords{equilibrium measure}
\address{University of Toronto, Toronto, Ontario, M5S2E4, Canada}
\email{bloom@math.toronto.edu}
\title{Large Deviation for outlying coordinates  in $\beta$ ensembles}

\maketitle 

\vspace{-4mm}

\begin{center}
May 28, 2013
\end{center}

\vspace{4mm}

\begin{abstract}
For $Y$ a subset of the complex plane, a $\beta$ ensemble is a sequence of probability measures $ Prob_{n,\beta ,Q}$ on $Y^n$ for $n=1,2,\ldots$ depending on  a positive real parameter $\beta$ and   a real-valued continuous function $Q$ on $Y.$ We consider the associated sequence of probability measures on $Y$ where the probability of a subset $W$ of $Y$  is given by the probability that at least one coordinate of $Y^n$ belongs to $W.$ With appropriate restrictions on $Y,Q$ we prove a large deviation principle for this sequence of probability measures. This extends a result of Borot-Guionnet to subsets of the complex plane and to $\beta $ ensembles defined with measures using a Bernstein-Markov condition.

\end{abstract}
\section{introduction}
$\beta$ ensembles are generalizations of the joint probability distributions of the eigenvalues  of the classical matrix ensembles. 
 They are defined as follows: let $Y$ be a closed subset of $\CC$, $Q$ a real-valued continuous function on $Y$ and $\beta>0.$ 
Consider the family of probability distributions $Prob_{n,\beta ,Q}$ for $n=1,2,...$ defined on $Y^n$ by:
  
\begin{equation}\label{i1}Prob_{n,\beta ,Q }=\frac{A_{n,\beta ,Q }(z)}{Z_{n,\beta ,Q }}d\tau(z)\end{equation}where
  \begin{equation}\label{i2}A_{n,\beta ,Q }(z):=|D(z_1,\ldots ,z_n)|^{\beta }\exp (-2n[Q(z_1)+\ldots +Q(z_n)]),\end{equation} 
 $$D(z_1,\ldots ,z_n)=\prod_{1\leq i<j\leq n}(z_i-z_j)$$ denotes the Vandermonde determinant, and
the normalizing constants $Z_{n,\beta ,Q}$  are given by:
 \begin{equation}\label{i3} Z_{n,\beta ,Q }(Y)= Z_{n,\beta ,Q }:=\int_{Y^n}A_{n,\beta ,Q }(z)d\tau(z).\end{equation}
 Here $$d\tau (z) =d\tau (z_1)\ldots d\tau (z_n)$$and $\tau$ is an appropriate measure on $Y.$ In particular, the support of $\tau$ is $Y.$  We assume that $R=2Q/\beta$ is of superlogarithmic (see (\ref{p0}))  growth if $Y$ is unbounded. 
 The existence of the integrals in the case of $Y$ unbounded is dealt with in the course of proving Theorem 7.3.

In this paper we will assume the measure $\tau$ satisfies a weighted Bernstein-Markov inequality (see section 3 and Hypothesis 3.9). Lebesgue measure in one or two dimensions satisfies  Hypothesis 3.9  but,  in addition,  there are more general  measures which also satisfy  Hypothesis 3.9. In fact there are discrete measures (i.e a countable linear combination of Dirac measures) which satisfy the hypothesis. 

For $Y=\RR,  d\tau=dx\: (\text {Lebesgue measure}), Q(x)=x^2/2$ and $ \beta=1,2,4$ we obtain the joint probability distribution of the eigenvalues of the classical matrix ensembles-respectively the Gaussian orthogonal, unitary and symplectic ensembles \cite{AGZ}.
In this context, the coordinates of a point $x=(x_1,\ldots,x_n)$ are referred to as eigenvalues.

The 2-dimensional version of these probability distributions occurs in the study of the Coulomb gas model (\cite{F}, \cite{HM}). In this model the parameter $\beta$  corresponds to the inverse temperature, $2Q$ to the confining potential, and the coordinates of a point are the positions of particles.

The two-dimensional version also occurs,  for $\beta=2$,  as the distribution of eigenvalues of random normal matrices (\cite{HM}).

These ensembles have been extensively studied,  primarily in the case $Y=\RR$ or $\CC$ and $d\tau$ as Lebesgue measure  (see  \cite{AGZ}, \cite {F}, \cite{HM}, \cite{HD} and the references given there). In particular, 
it is known that the normalized counting measure of a random point,  i.e. the random measure $\frac{1}{n}\sum_{j=1}^n\delta(z_j)$,  where $\delta$ denotes the Dirac measure,  converges, almost surely, $\text{ weak}^*$,  to a non-random measure with compact support -  the weighted equilibrium measure (see (\ref{p1})) . Furthermore a large deviation principle for the normalized counting measure  of a random point  is known \cite{AGZ}. The  large deviation principle has speed $n^2$ and  implies that  the probability of finding a positive proportion of the eigenvalues/particles/coordinates  outside a neighbourhood of the support of the weighted equilibrium measure is asymptotic to $e^{-n^2c}$ for some $c>0.$

In this paper we are concerned with the behaviour of a single eigenvalue/particle/coordinate. We will establish a large deviation principle for the coordinate $z_1\in Y$ under the probability distributions (\ref{i1}).
That is  we will prove a large deviation principle for the countable family of probability distributions on $Y$ given by, for $W$ an open subset of $Y$ 
\begin{equation}\label{i31}\psi_n (W)=Prob_{n,\beta ,Q}\{z_1\in W\}=\frac{1}{Z_{n,\beta ,Q }}\int_W\int_{Y^{n-1}}A_{n,\beta ,Q }(z)d\tau (z)\end{equation}\\
for $n=1,2,\ldots .$

The main results of this paper are Theorems 6.2 and 7.3 which show that
the sequence of probability measures $\psi_n$ satisfies  a large deviation principle with speed 
 $n$ and rate function
\begin{equation}\label{i4}\mathbb{J}_{Y,\beta ,Q}(z_1)=2Q(z_1)- \beta ( \int_Y\log|z_1-t|d\mu_{Y,\beta ,Q}+\rho)\end{equation}where $\mu_{Y,\beta ,Q}$
is the weighted equilibrium measure (see (\ref{p1})) and $\rho$ is a constant (see(\ref{p4})).

This l.d.p. implies that the probability of finding an eigenvalue/ \\particle/coordinate  outside a neighbourhood of the support of the weighted equilibrium measure is asymptotic to $e^{-nc}$ for some $c>0.$

Theorem  6.2 handles the case when $Y$ is compact and Theorem  7.3  the unbounded case. Theorem 6.2 is valid under the following hypotheses. (We let $S_R$ denote the support of the weighted equilibrium measure and $S_R^*$ the points where $\mathbb{J}_{n,\beta ,Q}(z_1)=0$)
\begin{enumerate}

\item Y is a regular set (in the sense of potential theory.)

\item(Hypothesis 3.9) $\tau$ satisfies the weighted BM inequality on any compact neighbourhood of $ S_R^*.$
\item(Hypothesis 6.1) $S_R^*=S_R.$
\end{enumerate}

Now (1) ensures that the weighted Green function (see (\ref{p2})) is continuous. (2) is always satisfied by Lebesgue measure on $\RR$ or $\CC$ (see section 3). (3) is the assumption that the rate function is strictly positive outside $S_R$. In \cite{BG} the corresponding assumption is referred to as control of large deviation.

The unbounded case, Theorem 7.3,  requires  additional assumptions. Theorem 7.3 is valid under the following assumptions: Condition (1) above,  is replaced by the hypothesis that the intersection of $Y$ with sufficiently large discs centre the origin is regular. (2) and (3) above must hold. The additional conditions are on the growth of $\tau$ and $R=2Q/{\beta}.$ Namely:\\
(4) (see(\ref{p0})) For some $b>0,
\lim_{|z|\to\infty}(R(z)-(1+b)\log|z|)=+\infty.$\\
(5) (Hypothesis 7.1) For some $ a>0 , \int_Yd\tau /|z|^a<+\infty.$

A key step is theorem 5.1 where we prove a result on the asymptotics of normalizing constants, proved in \cite{BG} for subsets of $\RR$ and used previously as an assumption in \cite{AGZ}. Our methods use polynomial estimates and potential theory but not the large deviation principle for the normalized counting measures of a random point. We will specifically use weighted potential theory (see \cite{ST}), since the probability distributions given by (\ref{i1}) are, in each variable, a power of the absolute value of a  weighted polynomial (see section 2).

These results extend  a  result of \cite{BG} to appropriate  subsets of the plane and, in addition,  the measure $\tau$ used to define the ensemble can be more general than Lebesgue measure: it need only satisfy the Bernstein-Markov condition given by Hypothesis 3.9.
 The rate function is independent of $\tau$ as long as $\tau$ satisfies Hypothesis 3.9.

We let $C,c$ denote positive constants which may vary from line to line. All measures are positive Borel measures.

\section{polynomial estimates}
We will list some basic results of weighted potential theory (see \cite{ST}).
 Let $Y$ be a  closed set in the plane and $R$ be a continuous (real-valued) function on $Y$.
If $Y$ is unbounded , $R$ is assumed  to be superlogarithmic i.e., for some $b>0$
\begin{equation}\label{p0}\lim_{|z|\to\infty}(R(z)-(1+b)\log|z|)=+\infty.\end{equation}

For $r>0$ we let $Y_r:=\{z\in Y||z|\leq r\}.$ 
We will assume that $Y_r$ is a regular set (in the sense of potential theory) for all $r$ sufficiently large. We recall that a compact  set is regular if it is regular for the exterior Dirichlet problem \cite{R} or equivalently if the unweighted Green function (i.e. $R\equiv 0$ in (\ref{p2}) below)
is continuous on $\CC.$

The weighted Green function of $Y$ with respect to $R$ (\cite{ST}, Appendix B)  is denoted by $V_{Y,R}$. It is defined by
\begin{equation}\label{p2} V_{Y,R}(z)=\sup\{u(z)|u\text{ is subharmonic on }\CC,u \leq R \text{ on }  Y\text{and, }\end{equation} $$u\leq \log^+|z| +C, \text{where } C \text{ is a constant depending on } u\}.$$
 
$V_{Y_r,R}$ is continuous for $r$ sufficiently large (\cite{S}, prop 2.16) since $Y_r$ was assumed to be regular for $r$ sufficiently large  and any regular compact set in the plane is locally regular (\cite{R}). 
Thus $V_{Y,R}$ is continuous since $V_{Y,R}=V_{Y_r,R}$ for $ r$ sufficiently large (\cite{ST}, Appendix B, lemma 2.2).

The \emph{logarithmic capacity} $c(E)$ of a Borel subset $E\subset\CC$ is defined by

$$\log c(E)=\sup_{\nu}\{\int\int\log|z-t|d\nu(z)d\nu(t)\}$$ over all Borel probability measures $\nu$ with compact support in $E.$

 Regular sets are of positive capacity and so $Y$ is of positive capacity. The function $e^{-R}$ is thus strictly positive on a set of positive capacity and so is  an admissible weight in the terminology of \cite{ST}.  Thus Chapter  I, Theorem 1.3 of \cite{ST} holds.

 The weighted equilibrium measure is denoted $\mu_{Y,R}$. It has compact support and  is the unique minimizer of
 \begin{equation}\label{p1}E(\nu )=-\int\int\log|z-t|d\nu (z)d\nu (t)+2\int R(z)d\nu (z)\end{equation}
over all measures $\nu \in \mathcal{M}(Y)$ where $\mathcal{M}(Y)$ denotes the probability measures on $Y$(\cite{ST}, Chapter  I, Theorem 1.3).

 Now, by (\cite{ST}, Appendix B, Lemma 2.4)
\begin{equation}\label{p3}V_{Y,R}(z)=\int_Y\log|z-t|d\mu_{Y,R}(t)+\rho\end{equation}
where $\rho$ is the Robin constant given by
\begin{equation}\label{p4}\rho=\lim_{|z|\rightarrow\infty}(V_{Y,R}(z)-\log|z|).\end{equation}
It is also known that $E(\mu_{Y,R})>-\infty$ and (\cite{ST}, Chapter  I, Theorem 3(d)) $$\rho=E(\mu_{Y,R})-\int_Y R(z)\mu_{Y,R}.$$

We let $S_R^*:=\{z\in Y|V_{Y,R}= R\}$ and $S_R:=\text{supp}(\mu_{Y,R})$.  $S_R$ and $S_R^*$ are compact and  in general $S_R\subset S_R^*$ (\cite{ST}, Chapter  I, Theorem 1.3(e)).   $S_R$ and $S_R^*$ depend on $Y$ although  our notation does not explicitly indicate this. However, for $r$ sufficiently large $S_R$ and  $S_R^*$ are independent of $Y_r.$

From (\ref{p3}) we have (\cite{ST}, Chapter  I,  Theorem 1.3(f)) \begin{equation}\label{p5}R(z)=\int_Y\log|z-t|d\mu_{Y,R}(t) +\rho\end{equation}for $z\in S_R.$\\

   We let $\mathcal{P}_n$ denote the space of  polynomials in the single variable  $z$ of degree $ \leq n$ and for $p\in\mathcal{P}_n$ we refer to $e^{  -nR(z)}p(z)$ as a \textit{weighted polynomial of degree} $n$ (the weight is the positive continuous function $ e^ { -R(z)}).$

 Now for $\beta >0$ we set \begin{equation}\label{pR}R(z)=\frac{2}{\beta}Q(z).\end{equation}
Note that $A_{n,\beta ,Q}(z)$ is, in each variable, of the form$$|e^{-nR(z)}p(z)|^{\beta}=|e^{-2nQ}p(z)|^{\beta}$$  and is thus  the absolute value of a weighted polynomial to the $\beta$ power. For this reason weighted potential theory can be used in the study of $\beta$ ensembles.\\

Equations (\ref{p51}) and (\ref{p52}) below are  known estimates for weighted polynomials.\\

By  (\cite{ST}, Chapter III,  Corollary 2.6)   the sup norm of a weighted polynomial is assumed on $S_R$. That is 
\begin{equation}\label{p51} ||e^{-nR}p(z)||_Y=||e^{-nR(z)}p(z)||_{S_R}\end{equation}\\
for all $p\in\mathcal{P}_n.$\\

By (\cite{ST}, Chapter I, Theorem 3.6)  we have, for $p$ a monic polynomial of degree $n$
 \begin{equation}\label{p52}||e^{-nR(z)}p(z)||_{S_R}\geq e^{-n\rho }.\end{equation}\\

For $G$ a subset of $\mathcal{M}(Y)$ we will use the notation $$\tilde{G}_n=\{z=(z_1,\ldots,z_n)\in Y^n\Big{|}\frac{1}{n}\sum_{j=1}^n\delta (z_j)\in G\}.$$
Here $\delta$ denotes the Dirac delta measure.\\

We now restrict to the case that $Y$ is compact and we will use the notation $K$ in place of $Y.$
\begin{lemma}\label{lemma1}
Given $\epsilon >0$, there is a neighbourhood $G$ of $\mu_{K,R}$ in $\mathcal{M}(K)$ (with the weak* topology)  such that for $(z_1,\ldots ,z_n)\in \tilde{G}_n$ we have $$||e^{-nR(t)}(t-z_1)\ldots  (t-z_n)||_K \leq e^{-n(\rho -\epsilon )} .$$\end{lemma}
\begin{proof}
The proof will be by contradiction. If not, for some $\epsilon >0$, no such $G$ exists. Thus there exists a sequence of $n_s$-tuples $(z^s_1,\ldots ,z_{n_s}^s)$ for $s=1,2,\ldots$ with
 \begin{equation}\label{p6}\lim_s\frac{1}{n_s}\sum_{j=1}^{n_s}\delta (z_j^s)=\mu_{K,R}\end{equation}
weak*, but, using (\ref{p51}),

 \begin{equation}\label{p7}||e^{-n_sR(t)}(t-z_1^s)\ldots (t-z_{n_s}^s)||_{S_R}\geq e^{-n_s\rho}e^{n_s\epsilon}.\end{equation}
Taking logarithms this may be rewritten as:
\begin{equation}\label{p8}\Big\|-R(t)+\frac{1}{n_s}\sum_{j=1}^{n_s}\log |t-z_j^s|\Big\|_{S_R}\geq -\rho +\epsilon.\end{equation} 
 It follows from (\ref{p6}) that, using  (\cite{ST}, Chapter 0, Theorem 1.4), we have
\begin{equation}\label{p9}\limsup_s \frac{1}{n_s}\sum_{j=1}^{n_s}\log |t-z_j^s|\leq\int\log |t-\xi |d\mu_{K,R}(\xi),\end{equation}
since $\xi\rightarrow \log |t-\xi |$ is uppersemicontinuous. Now, (\ref{p9}) is a \emph{pointwise} $\limsup$ and to obtain a contradiction  we will require a uniform $\limsup$  for $t\in S_R$. 

Since  $\xi\rightarrow \log |t-\xi |$ is subharmonic,
by Hartogs' lemma, (\cite{K}, Theorem 2.6.4) and (\ref{p5}) we have 
 \begin{equation}\label{p10}\limsup_s \frac{1}{n_s}\sum_{j=1}^{n_s}\log |t-z_j^s|\leq R(t)-\rho +\epsilon/2,\end{equation} uniformly on $S_R$. That is, for $s$ sufficiently large,
 \begin{equation}\label{p11}\Big\|-R(t)+ \frac{1}{n_s}\sum_{j=1}^{n_s}\log |t-z_j^s|\Big\|_{S_R}\leq -\rho +\epsilon/2,\end{equation} which contradicts (\ref{p8}).

\end{proof}
The following is a corollary to the proof of Lemma \ref{lemma1} and will be used in the proof of Theorem 6.2.
\begin{corollary}
Given $\epsilon >0$, there is a neighborhood $G$ of $\mu_{K,R}$ in $\mathcal{M}(K)$ (with the weak* topology)  such that for $(z_2,\ldots ,z_n)\in \tilde{G}_{n-1}$ and $w\in K$ then
 \begin{equation}\label{p111}(\frac{1 }{n-1}\sum_{j=2}^n\log |w-z_j|-\int\log|w-t|d\mu_{K,R}(t))\leq\epsilon.\end{equation}
\end{corollary}

\section{Bernstein-Markov inequalities}
 
Although we ultimately will only need Bernstein-Markov inequalities for subsets of $\CC$ we will consider these inequalities in $\CC^N$ as the $\CC^2$ case is used in the proofs.
Recall that all measures are positive Borel measures.
\begin{definition}
Let $\tau$ be a  measure  on a compact set $K\subset\CC^N.$  We say $\tau$ satisfies the  Bernstein-Markov (BM) inequality   if, for all $\epsilon >0$, there exists $C>0$ (independent of $n,p$) such that
 \begin{equation}\label{bm1}||p||_K\leq C(1+\epsilon)^n\int_K |p|d\tau,\end{equation}
for all $p\in\mathcal{P}_n^N, $ where here $\mathcal{P}_n^N$ denotes the holomorphic polynomials of total degree $\leq n$ on $\CC^N$. \end{definition}.
\begin{definition}
Let $\tau$ be a  measure  on a compact set $K\subset\CC^N$ and $R$ a real-valued continuous function on $K.$ We say $\tau$ satisfies the weighted Bernstein-Markov (BM) inequality for the weight  $e^{-R}$ if, for all $\epsilon >0$, there exists $C>0$ (independent of $n,p$) such that
 \begin{equation}\label{bm2}||e^{-nR}p||_K\leq C(1+\epsilon)^n\int_Ke^{-nR}|p|d\tau,\end{equation}
for all $p\in\mathcal{P}_n^N.$\end{definition}
We say that $\tau$ satisfies a$\emph{ strong BM inequality}$ if it satisfies the weighted BM inequality for all weights $e^{-R}.$\\

It  is known (\cite{StT}, proof of Theorem 3.4.3 ) that if $\tau$ satisfies (\ref{bm1})  or (\ref{bm2})  then it also satisfies an $ L^{\beta}$ version of that inequality (with, possibly,  a different constant $C$). For example, from (\ref{bm2})
\begin{equation}\label{bm3}||e^{-nR}p||_K\leq C(1+\epsilon )^n(\int_K(e^{-nR}|p|)^{\beta}d\tau)^{\frac{1}{\beta}}.\end{equation}
 
 The following lemma will be used in the proof of Theorem 5.1.
 \begin{lemma}Let $K\subset\CC$ be regular, $R$ a continuous real-valued continuous function on $K$  and $\tau$ be a  measure on $K$ satisfying the weighted BM inequality for the  weight $e^{-R}.$ Then for all monic polynomials of degree $n$ \begin{equation}\label{p14}\int_K|e^{-2nQ}p^{\beta}|d\tau\geq C(1+\epsilon )^{-n\beta }e^{-n \rho\beta}.\end{equation}\begin{proof}
Recall that $R=2Q/{\beta}$ and combine  (\ref{bm2})  with (\ref{p51})  and (\ref{p52}). 
\end{proof}
\end{lemma}
 
Conditions on a measure  to satisfy the BM inequality for compact sets $K\subset \CC$ are extensively studied in \cite{StT}. The measures  termed \emph{regular} in \cite{StT} coincide with measures satisfying the BM inequality on  regular compact sets but 
 the BM inequality is a more stringent condition in general (see \cite{StT},  example 3.5.3).  
\begin{remark}Suppose that $\tau$ satisfies the weighted  BM inequality for $e^{-R}$ on $S_R.$ Then it follows from (\ref{bm2}) and    (\ref{p51})  that $\tau$ satisfies the weighted BM inequality for $e^{-R}$ on any compact set $K\supset S_R.$ \end{remark}
Hence any  measure  $\tau$ which satisfies the weighted BM inequality for $e^{-R}$ on $S_R$  will satisfy Hypothesis 3.9 below i.e.
$\tau$ satisfies the weighted BM inequality on any compact neighbourhood of $S_R.$  Now for \emph{every} compact set $K\subset\CC$ there exists a discrete  measure ( i.e. a  countable linear combination of Dirac $\delta$ measures ) on $K$ which satisfies the strong BM inequality (see \cite{BL}, Corollary 3.8 - the construction given there is, in fact, valid for every compact set). Applying this construction to $S_R$  gives a discrete measure which satisfies Hypothesis 3.9.
We will show that Lebesgue measure on $\RR$ or $\CC$ also satisfies Hypothesis 3.9.
 
The next proposition gives a convenient  sufficient (but by no means necessary)  condition for a measure $\tau$ to satisfy the BM inequality on a regular compact set $K\subset\CC^N$ (for the definition of regular compact set in $\CC^N,N>1$ see (\cite{K}, Chapter 5).
\begin{proposition} 
  ( \cite{BL2}, Theorem 2.2)   Let $K\subset\CC^N$ be compact and regular. Let $\tau$ be a measure on $K$ such that  
for some $T>0$ and each $z_0\in K$ there exists $r(z_0)>0$ such that
\begin{equation}\label{bm4}\tau\{B(z_0,r)\}\geq r^T\text{ for }  r\leq r(z_0)\end{equation}
where $B(z_0,r)$ denotes the ball centre $z_0$, radius $r.$ Then $\tau$ satisfies the BM inequality on $K.$\end{proposition}

We also have:
\begin{proposition}( \cite{BB}, Theorem 3.2) Let $ K\subset\RR^N\subset\CC^N.$ Then if $\tau$ satisfies the  BM inequality  on $K$ it satisfies the strong BM inequality  on $K.$

\end{proposition}
\begin{corollary}
 Let $K$ be the closure of its interior in $\CC$ and have $\mathcal{C}^1$ boundary. Then  Lebesgue planar measure satisfies the strong BM  inequality on  $K.$
\begin{proof}Consider $K\subset\CC\cong\RR^2\subset\CC^2.$ Now, $K$ is regular as a subset of $\CC^2$ (this follows, for example from the 
accessibility condition of Plesniak,\cite{K}, Chapter 5) and any measure satisfying a strong BM inequality on $K$ as a subset of $\CC^2$ satisfies a strong BM inequality on $K$ as a subset of $\CC.$ Thus the result follows from Propositions 3.5 and 3.6.\end{proof}
\end{corollary}
\begin{corollary}Let $K\subset\RR$ be a finite union of disjoint closed intervals.Then Lebesgue measure on $\RR$ satisfies the strong BM inequality on $K.$ 
\begin{proof}$K$ is regular as a subset of $\CC$ (\cite{StT})  so it follows from Propositions 3.5 and 3.6 that Lebesgue measure on $\RR$ satisfies the strong BM inequality on $K.$
\end{proof}
\end{corollary}
We will use the following hypothesis in the next lemma:
\begin{hypothesis}
The  measure $\tau$ satisfies the weighted BM inequality for $e^{-R}$ on any compact neighbourhood of $S_R^*.$\end{hypothesis}
\begin{remark}  Every compact set $K\subset\CC$ has a neighbourhood basis of compact sets which are the closure of their interiors and have $\mathcal{C}^1$ boundary. Thus by Corollary 3.7,  Lebesgue measure  on  $\CC$ satisfies Hypothesis 3.9. Similarly, since every compact set $K\subset\RR$ has a neighbourhood basis of compact sets which are a finite disjoint union of closed intervals,  so by Corollary 3.8,  Lebesgue measure on $\RR$ satisfies Hypothesis 3.9.    
\end{remark}

\begin{lemma} Let $K\subset\CC$ be regular, $R$ a  real-valued continuous function on $K$  and $\tau$ be a  measure on $K$ satisfying Hypothesis 3.9. Let $N$ be a compact neighbourhood of $S_R^*.$ Then there is a constant $c>0$ (independent of $n,p$) such that:$$\int_K|e^{-nR}p|^{\beta}d\tau\leq (1+O(e^{-nc}))\int_N|e^{-nR}p|^{\beta}d\tau$$
for all $p\in\mathcal{P}_n.$\end{lemma}
\begin{proof}
We first normalize the polynomial so that $||e^{-nR}p||_{S_R}=1.$ To prove the theorem it will suffice to show that
\begin{equation}\label{l11}\int_{K\setminus N}|e^{-nR(z)}p(z)|^{\beta}d\tau\leq Ce^{-nc},\end{equation} for some constants $C,c>0$ and
\begin{equation}\label{l12}\liminf_n \Big(\int_K|e^{-nR}p|^{\beta}d\tau \Big)^{1/n}\geq 1.\end{equation} 

We will use the estimate (\cite{ST}, Appendix B, Theorem 2.6 (ii)) 
\begin{equation}\label{l13}|e^{-nR(z)}p(z)|\leq||e^{-nR(z)}p(z)||_{S_R}
\exp(n(V_{K,R}(z)-R(z)) \end{equation}for $z\in K.$\\

Since $N$ is a compact neighbourhood of $S_R^*$, for $z\in K\setminus N$, and some constant $b>0$
\begin{equation}\label{l17}V_{K,R}(z)-R(z)\leq -b<0\end{equation}
so $|e^{-nR(z)}p(z)|\leq e^{-nb}$ for $z\in K\setminus N.$\\
Thus,\begin{equation}\label{l18}\int_{K\setminus N}|e^{-nR(z)}p(z)|^{\beta}d\tau\leq Ce^{-nb\beta }\end{equation}
for constants $C,b>0.$\\
Now,\begin{equation}\label{l19}\int_K|e^{-nR(z)}p(z)|^{\beta}d\tau\geq \int_N|e^{-nR(z)}p(z)|^{\beta}d\tau\end{equation}
and since $\tau$ satisfies the weighted BM condition on $N$ for $e^{-R}$  the right hand side in (\ref{l19}) is, for any $\epsilon >0$ and some $C>0$
\begin{equation}\label{l20}\geq ||e^{-nR}p||_{S_R}^{\beta}Ce^{-\epsilon n}=Ce^{-\epsilon n},\end{equation}
establishing (\ref{l12}) and the result.

\end{proof}
 Theorem 7.2  establishes  a version of this lemma for unbounded sets.

\section{Johansson large deviation}

We will not use the l.d.p. for the normalized counting measure of a random point but a weaker result whose utility was shown by Johansson \cite{J}.

Consider \begin{equation}\label{j1}A_{n,\beta ,Q}(z):=|D(z_1,\ldots ,z_n)|^{\beta }\exp (-2n[Q(z_1)+\ldots +Q(z_n)])\end{equation}where $\beta >0$,  $D$  denotes the Vandermonde determinant $$D(z_1,\ldots ,z_n)=\prod_{1\leq i<j\leq n}(z_i-z_j),$$
and $Q$ is a continuous, real-valued function on a regular compact set $K.$
Let $\tau$ satisfy the weighted BM inequality on $K$ for $e^{-R}$.
Let \begin{equation}\label{j2} Z_{n,\beta ,Q}:=\int_{K^n}A_{n,\beta,Q }(z)d\tau (z)\end{equation} where $$d\tau (z) =d\tau (z_1)\ldots d\tau (z_n).$$
We define a probability measure on $K^n$ for $n=1,2\ldots$ by\begin{equation}\label{j3}Prob_{n,\beta ,Q }=\frac{A_{n,\beta ,Q }(z)}{Z_{n,\beta ,Q }}d\tau(z).\end{equation}
 We obtain a collection of probability measures which we refer to as a $\beta$ ensemble on $K.$

We let, for $\nu\in\mathcal{M}(K)$\begin{equation}\label{j4}E_{\beta } (\nu )=-\beta /2\int\int\log|x-y|d\nu (x)d\nu(y) +2\int Q(x)d\nu(x).\end{equation}
Then\begin{equation}\label{j5}E_{\beta }(\nu )=\frac{\beta}{2}E(\nu ).\end{equation}
Thus the unique minimizer in $\mathcal{M}(K)$ of  $E_{\beta}(\nu)$ is $\mu_{K,R}$ for which we will also use the notation $\mu_{K,\beta ,Q}.$

The functional $E_{\beta}$ is a special case of the functionals $ E_Q(\mu )$ studied in \cite{BLW} and it is analogous to the functionals $E(\nu )$ and $E^Q(\nu)$ considered in (\cite{BA},equations (2.9) and (5.4)).The following three propositions are therefore special cases of results in \cite{BLW} and the proofs are also analogous to  proofs  in \cite{BA}.

\begin{proposition}
$$\lim_n\frac{1}{n^2}\log Z_{n,\beta ,Q }=-E_{\beta }(\mu_{K,\beta ,Q})=\lim_n\frac{1}{n^2}\log\sup_{K^n}A_{n,\beta ,Q}(z)$$
\begin{proof}See \cite{BLW},Proposition 4.15 and  Proposition 3.2 or \cite{BA}, Theorem 3.9 and Corollary 3.6.  
\end{proof}
\end{proposition}
 We note that the two quantities in the  equality on the right  of Proposition 4.1  do not depend on the  measure $\tau.$ 
 
Let $\log\gamma :=-E_{\beta }(\mu_{K,\beta ,Q}).$ Then $\gamma >0$ and let $\eta$ , $0 < \eta <\gamma$ be given.
We define $$B_{\eta, n,\beta }^Q:=\{z\in K^n|A_{n,\beta ,Q}^\frac{1}{n^2}(z)\leq \gamma -\eta \}.$$
Then we have the following result, which we refer to as a Johansson large deviation result:
\begin{proposition}$$Prob_{n,\beta ,Q }(B_{\eta ,n,\beta }^Q)=\frac{1}{Z_{n,\beta ,Q }}\int_{B^Q_{\eta ,n,\beta }}A_{n,\beta ,Q }(z)d\tau (z) \leq(1-\frac{\eta }{2\gamma})^{n^2}\tau (K)^n$$ for all n sufficiently large.\end{proposition}
\begin{proof}
See \cite{BLW}, Proposition 4.15 or \cite{BA}, Theorem 4.1.
\end{proof}
\begin{proposition}Let $G$ be a neighbourhood of $\mu_{K,\beta ,Q}$ in $\mathcal{M}(K).$Then 
$$\frac{1}{Z_{n,\beta ,Q}}\int_{K^n\setminus\tilde{G}_n}A_{n,\beta ,Q}(z)d\tau\leq O(e^{-cn^2})$$
for some $c>0.$
\begin{proof}It follows (see \cite{BLW}, claim (6.16)  or Proposition 7.3 of \cite{BA}) that for some $\eta >0$,$$B^Q_{\eta ,n,\beta }\supset K^n\setminus\tilde{G}_n$$
for all $n$ sufficiently large, and so the result then follows from Proposition 4.2.\end{proof}
\end{proposition}

\section{the normalizing constants}
Proposition 4.1 gives an asymptotic result for the normalizing constants $ Z_{n,\beta ,Q }$. Note that if we have a sequence of continuous functions $\{Q_n\}$ converging uniformly to $Q$ on $K$, then $\frac{1}{n^2}\log Z_{n,\beta ,Q }$ and
$\frac{1}{n^2}\log Z_{n,\beta ,Q_n}$ have the same limit, and in particular,  this is true for $\frac{1}{n^2}\log Z_{n,\beta ,\frac{nQ}{n-1}}.$
We will,  however,  need a sharper result, namely:
\begin{theorem}Let $K$ be a regular compact subset of $\CC$, $Q$ a real-valued continuous function on $K$ and $\tau$ a   measure on $K$ satisfying the  weighted BM inequality for $e^{-R}$. Then$$\lim_{n\rightarrow\infty}\Big(\frac{Z_{n,\beta ,Q}}{Z_{n-1,\beta ,\frac{nQ}{n-1}}}\Big)^{\frac{1}{n}}=e^{-\rho\beta}.$$
\end{theorem}  
\begin{proof}
We will first prove that \begin{equation}\label{e1}\liminf_{n\rightarrow\infty}\Big(\frac{Z_{n,\beta ,Q}}{Z_{n-1,\beta ,\frac{nQ}{n-1}}}\Big)^{\frac{1}{n}}\geq e^{-\rho\beta}.\end{equation}
Now $ Z_{n,\beta ,Q}=\int_{K^n}A_{n,\beta ,Q }(z)d\tau (z).$ 
We regard $A_{n,\beta ,Q}(z)$ as a function of $z_1$ and we  apply Lemma 3.3
to the integral in the $z_1$ variable to obtain
\begin{equation}\label{e3}Z_{n,\beta ,Q}\geq C(1+\epsilon )^{-n\beta }e^{-n\rho\beta}\int_{K^n}|D(z_2,\ldots ,z_n)|^{\beta }e^{-2n[Q(z_2)+\ldots+Q(z_n)]}d\tau (z_2)\ldots d\tau (z_n)\end{equation}$$=C(1+\epsilon )^{-n\beta}e^{-n\rho\beta }Z_{n-1,\beta ,\frac{nQ}{n-1}}.$$
Since $\epsilon >0 $ is arbitrary, the estimate on the lower limit follows.

To complete the proof we must show
  \begin{equation}\label{e4}\limsup_{n\rightarrow\infty}\Big(\frac{Z_{n,\beta ,Q}}{Z_{n-1,\beta ,\frac{nQ}{n-1}}}\Big)^{\frac{1}{n}}\leq e^{-\rho\beta}.\end{equation}
Let  $\tilde{G}_{n-1}$ be a subset of $K^{n-1}$ determined as follows: Given $\epsilon > 0$ choose a neighbourhood $G$ of $\mu_{K,\beta ,Q}$  so that Lemma 2.1  holds. Recall that
$$\tilde{G}_{n-1}=\{(z_2,\ldots ,z_n)\in K^{n-1}|\frac{1}{n-1}\sum_{j=2}^n\delta (z_j) \in G\}.$$
 We  write the integral for $Z_{n,\beta ,Q}$ as a sum of two integrals $$Z_{n,\beta ,Q}=I_1+I_2$$where
$$I_1=\int_K\int_{K^{n-1}\setminus \tilde{G}_{n-1}}A_{n,\beta ,Q}(z)d\tau (z)$$ and
$$I_2=\int_K\int_{\tilde{G}_{n-1}}A_{n,\beta ,Q}(z)d\tau (z).$$
Now,\begin{equation}\label{e6}I_1=\int_K\prod_{j=2}^n|z_1-z_j|^{\beta }e^{-2nQ(z_1)}d\tau (z_1)\times \end{equation}
$$\int_{K^{n-1}\setminus\tilde{G}_{n-1}}|D(z_2,\ldots,z_n)|^{\beta }e^{-2n[Q(z_2)+\ldots+Q(z_n)]}d\tau (z_2)\ldots d\tau (z_n).$$
Since $K$ is compact, the first factor is $O( C^n)$ for some  $C>0.$  The integrand in the second factor differs from $A_{n-1,\beta , Q}(z_2,\ldots,z_n)$ by  a factor of $O(C^n)$ for some $C>0$ so the second factor is $O(e^{-cn^2})Z_{n-1,\beta ,Q}$ by Proposition 4.3.
 Since 
$$Z_{n-1,\beta , Q}=O(C^n)Z_{n-1,\beta ,\frac{nQ}{n-1}},$$we may conclude that $$I_1=O(e^{-cn^2})Z_{n-1,\beta,\frac{n Q}{n-1}}.$$
Also \begin{equation}\label{e5}I_2=\int_K\prod_{j=2}^n|z_1-z_j|^{\beta }e^{-2nQ(z_1)}d\tau (z_1)\times\end{equation}
$$\int_{\tilde{G}_{n-1}}|D(z_2,\ldots,z_n)|^{\beta }e^{-2n[Q(z_2)+\ldots+Q(z_n)]}d\tau (z_2)\ldots d\tau (z_n).$$
We will need a more precise estimate on the first factor than the one used in (\ref{e6}). Since we are integrating over $\tilde{G}_{n-1}$ we may use Lemma 2.1 on the first factor to see that it is
 $\leq C e^{-n\beta(\rho-\epsilon )}.$  The second factor is $\leq Z_{n-1,\beta ,\frac{nQ}{n-1}}$
since $\tilde{G}_{n-1}$ is a subset of $K^{n-1}.$
Hence, given any $\epsilon > 0$,  we have for some $c,C>0$ $$I_1+I_2\leq Ce^{-n\beta (\rho -\epsilon)}Z_{n-1,\beta ,\frac{nQ}{n-1}}+O(e^{-cn^2})Z_{n-1,\beta ,\frac{nQ}{n-1}}$$and
(\ref{e4}) follows.
\end{proof}

\section{large deviation}
Given a separable, complete metric space $X,$  a sequence of probability measures $\{\sigma_n\}$ on $X$ is said to satisfy a
large deviation principle  with speed $n$ and rate function $\mathbb{J}(x)$ if $\mathbb{J}:X\rightarrow [0,\infty ]$ is lower-semicontinuous, $\{x\in X|\mathbb{J}(x)\leq l\}$ is compact for $l\geq 0$ and\\
(i) For all closed sets $F\subset X$ we have
$$\limsup_n\frac{1}{n}\log \sigma_n(F)\leq-\inf_{x\in F}\mathbb{J}(x).$$
(ii) For all open sets $G\subset X$ we have
$$\liminf_n\frac{1}{n}\log \sigma_n(G)\geq-\inf_{x\in G}\mathbb{J}(x).$$
If $X$ is compact by (\cite{AD}, Theorem 4.1.11) to establish the l.d.p. it suffices to show that, for all $x\in X$
\begin{equation}\label{lr}-\mathbb{J}(x)=\lim_{\epsilon\rightarrow 0}\lim_n\frac{1}{n}\log \sigma_n(B(x,\epsilon ))\end{equation}
where $B(x,\epsilon)$ is the ball centre $x$, radius $\epsilon.$
If $X$ is non-compact, to establish the l.d.p. there is an additional condition required, termed \textit{exponential tightness,} namely:
For all $r>0$ there is a compact set  $X_r\subset X$ with $\sigma_n(X_r)\leq -r.$

 Let $K\subset\CC$ be a regular compact set and $Q$  a real - valued continuous function on $K$. Let $\tau$ be a  measure on $K$ which satisfies the weighted BM inequality for $ e^{-R}$
   (recall that $R=2Q/\beta$). Given the probability measures $ Prob_{n,\beta ,Q}$ forming  a $\beta$ ensemble on $K$, we consider the countable family of probability measures $\{\psi_n\},n=1,2,...$ on $K$ given as follows:
For $W$ an open subset of $K$

\begin{equation}\label{l0}\psi_n(W)=Prob_{n,\beta ,Q}\{z_1\in W\}=\frac{1}{Z_{n,\beta ,Q }}\int_W\int_{K^{n-1}}A_{n,\beta ,Q }(z)d\tau (z).\end{equation}

To prove the l.d.p., we will assume that $\tau$ satisfies Hypothesis 3.9. We will also use an additional assumption.
\begin{hypothesis}
\begin{equation}S_R=S_R^*.\end{equation}
\end{hypothesis}
Now, to prove the l.d.p. we will use equation (\ref{lr}) and so we will have to estimate
 $$\inf_{\{W|z_1\in W\}}\lim_n\frac{1}{n}\log\psi_n(W).$$ We will be able to do so for $z_1\in K\setminus S_R^*$ and $z_1\in S_R$. Thus under  Hypothesis 6.1  the l.d.p. will be complete.

For $\beta$ ensembles on $\RR$ ,  $\tau $ the Lebesgue measure on $\RR$ and $Q$ real analytic, then 
results of \cite{KM}  show that Hypothesis 6.1  holds for all but at most  countably many values of $\beta.$

   A class of examples where Hypothesis 6.1 holds are provided by the computations in 
(\cite{ST}, Chapter  IV, section 6). We suppose $R(z)$ is invariant under rotations of the plane, that  $R(t)$ is differentiable on $ (0,\infty)$  and satisfies $tR'(t)$ and $R(t)$ are increasing on $ (0,\infty)$. Then $S_R$ is a disc of radius $T_0$  given by the  solution of $T_0R'(T_0)=1.$ If we take $K$  a disc of radius $>T_0$,  then $V_{K,R}(z)=\log|z|+R(T_0)-\log T_0$ while $R(z)>V_{K,R}(z)$ for $|z|>T_0$
 so that $S_R^*=S_R$. $R(z)=|z|^2$ provides a specific example.

  It is simple to construct examples where Hypothesis 6.1 does not hold. For example, let $K$ be a large disc with $S_R^*\subset \text{interior}(K).$ Replacing $R$ by $R'=V_{K,R}$ then $V_{K,R'}=V_{K,R}$  so $ S_{R'}=S_R$  but $S_{R'}^*=K.$

\begin{theorem}
Let $K$ be a regular, compact subset of $\CC ,$  $Q$ a real-valued continuous function on $K$ and $\tau$ a  measure with supp($\tau)= K$ and  satisfying Hypotheses 3.9 and 6.1. Let $Prob_{n,\beta ,Q}$ be a $\beta$ ensemble on $K.$ Then the sequence of measures $\psi_n$, given by $\psi_n(W)=Prob_{n,\beta ,Q}\{z_1\in W\}$  for $W$ an open subset of $K$  satisfies a l.d.p. with speed $n$ and rate function
\begin{equation}\label{l0}\mathbb{J}_{K,\beta ,Q}(z_1)=\beta (R(z_1)-V_{K,R}(z_1))=2Q(z_1)-\beta V_{K,R}(z_1).\end{equation}  \end{theorem}
\begin{proof}
Using (\ref{p3}) we have $$\mathbb{J}_{K,\beta ,Q }(z_1)=2Q(z_1)-\beta (\int_K\log|z_1-t|d\mu_{K,\beta,Q}(t)+\rho).$$ Note that $\mathbb{J}_{K,\beta ,Q}(z_1)=0$  for $z_1\in S_R.$\\
Also since $V_{K,R}$ is continuous, the function
\begin{equation}\label{lc}z_1\rightarrow\int_K\log|z_1-t|d\mu_{K,\beta ,Q}(t)\end{equation} is continuous  and so is the rate function.

We will use  equation (\ref{lr}). Let $W$ be an open subset of $K$. We will estimate the probability that a point $w\in W.$\\
 By definition, 
\begin{equation}\label{l1} Prob_{n,\beta ,Q }\{w\in W\}=\frac{1}{Z_{n,\beta ,Q }}\int_W\int_{K^{n-1}}A_{n,\beta ,Q }(w,z_2,...,z_n)d\tau .\end{equation}
We will separately estimate $\limsup_n$ and  $\liminf_n$ of $$1/n\log(Prob_{n,\beta ,Q}\{w\in W\})$$ 
We begin with the $\limsup.$
We write the above integral as a sum of two integrals (where $G$ is an open neighbourhood of $\mu_{K,\beta ,Q}\subset \mathcal{M}(K)$  which is to be specified).$$\frac{1}{Z_{n,\beta ,Q }}\int_W\int_{K^{n-1}}A_{n,\beta ,Q }(w,z_2,...z_n)d\tau=H_1+H_2.$$
$$H_1=\frac{1}{Z_{n,\beta ,Q}}\int_W\int_{K^{n-1}\setminus\tilde{G}_{n-1}}A_{n,\beta ,Q}(w,z_2,...z_n)d\tau$$and
$$H_2=\frac{1}{Z_{n,\beta ,Q}}\int_W\int_{\tilde{G}_{n-1}}A_{n,\beta ,Q}(w,z_2,...z_n)d\tau.$$
  Similar to the estimate for $I_1,$ (see (\ref{e6}))  we have  $H_1=O(e^{-cn^2}).$\\
To estimate $H_2$  we now proceed as in \cite{AGZ}.
Let
\begin{equation}\label{l2}h_n:=\frac{Z_{n,\beta ,Q}}{Z_{n-1,\beta ,\frac{nQ}{n-1}}}.\end{equation}
Then
 \begin{equation}\label{l3} H_2=\frac{1}{h_nZ_{n-1,\beta ,\frac{nQ}{n-1}}}\int_W\int_{\tilde{G}_{n-1}}A_{n,\beta ,Q }(w,z_2,...z_n)d\tau .\end{equation}
Now, given $\epsilon >0$, using Corollary  2.2, let $G$ be a neighbourhood of $\mu_{K,\beta ,Q}\subset \mathcal{M}(K)$ so that for $w\in W$ and $ (z_2,\ldots,z_n)\in\tilde{G}_{n-1} $ then
 \begin{equation}\label{l4}\beta\Big (\frac{1 }{n-1}\sum_{j=2}^n\log |w-z_j|-\int\log|w-t|d\mu_{K,\beta ,Q}(t)\Big )\leq\epsilon.\end{equation}
Taking exponentials
\begin{equation}\label{l41} \prod_{j=2}^n|w-z_j|^{\beta }\leq e^{(n-1)(\epsilon+\beta\int\log |w-t|d\mu_{K,\beta ,Q}(t))}.\end{equation}
Also, since ${\tilde{G}_{n-1}}\subset K^{n-1}$

 \begin{equation}\label{l8}\frac{1}{Z_{n-1,\beta ,\frac{nQ}{n-1}}}\int_{\tilde{G}_{n-1}}|D(z_2,\ldots ,z_n)|^{\beta }e^{-2n[Q(z_2)+\ldots +Q(z_n)]}d\tau (z_2)\ldots d\tau (z_n)\leq 1.\end{equation}

Using Theorem 5.1 to estimate $h_n$,  the  inequalities in (\ref{l41}) and (\ref{l8}), and recalling that $\tau(W)\neq 0$ since the support of $\tau$ is $K$, we have
\begin{equation}\label{l10}\limsup_n1/n\log (Prob_{n,\beta ,Q}\{w\in W\})\end{equation}
$$\leq \beta\rho + sup_{w\in W}(\beta \int_K\log|w-t|d\mu_{K,\beta ,Q}-2Q(w)).$$
and using (\ref{lc}) we have 
\begin{equation}\label{l101}\inf_{\{W|z_1\in W\}}\limsup_n1/n\log (Prob_{n,\beta ,Q}\{w\in W\})\end{equation}
$$\leq \beta\rho + \beta \int_K\log|z_1-t|d\mu_{K,\beta ,Q}-2Q(z_1).$$

Now we must deal with the $\liminf.$ First we consider $z_1\notin S_R.$  Let $W$ be a neighbourhood of $z_1$ with $\overline{W}\cap S_R=\emptyset,$ and
let $N$ be a compact neighbourhood of $S_R$ such that $N\cap \overline{W}=\emptyset .$
Now $A_{n,\beta ,Q }(z)$  is  in each variable of the form $e^{-2nQ}|p|^{\beta}$ for a polynomial $p\in\mathcal{P}_n$. Hence by repeated use of Lemma 3.11  we have,
 
\begin{equation}\label{l22}\int_{K^n}A_{n,\beta ,Q}(z)d\tau (z)\leq (1+O(e^{-cn}))\int_{N^n}A_{n,\beta ,Q}(z)d\tau (z),\end{equation}
or
\begin{equation}Z_{n,\beta ,Q}(K)\leq (1+O(e^{-cn}))Z_{n,\beta ,Q}(N)\end{equation}
where $K$ and $N$ indicate the sets on which the normalizing constants are evaluated.
Similarly,
\begin{equation}\label{l23}Z_{n-1,\beta ,\frac{nQ}{n-1}}(K)\leq (1+O(e^{-cn})) Z_{n-1,\beta ,\frac{nQ}{n-1}}(N).\end{equation}
Now
$$\frac{1}{Z_{n,\beta ,Q} (K)}\int_W\int_{K^{n-1}}A_{n,\beta ,Q }(w,z_2,...z_n)d\tau$$
$$\geq\frac{1}{Z_{n,\beta ,Q}(K) }\int_W\int_{N^{n-1}}A_{n,\beta ,Q }(w,z_2,...z_n)d\tau$$

and using (\ref{l22}), to obtain the lower bound it suffices to estimate:
$$\frac{1}{Z_{n,\beta ,Q}(N) }\int_W\int_{N^{n-1}}A_{n,\beta ,Q }(w,z_2,...z_n)d\tau  =$$
$$\frac{1}{h_n(N)Z_{n-1,\beta ,\frac{nQ}{n-1}}(N)}\int_W\int_{N^{n-1}}A_{n,\beta ,Q }(w,z_2,...z_n)d\tau.$$
Given $\epsilon >0$ let $F$ be a neighbourhood of $\mu_{K,\beta ,Q}$ in $\mathcal{M}(N)$ such that for $w\in W$ and $(z_2,\ldots,z_n)\in \tilde{F}_{n-1}$ we have, 
\begin{equation}-\epsilon\leq \beta \Big(\frac{1 }{n-1}\sum_{j=2}^n\log |w-z_j|-\int\log|w-t|d\mu_{K,\beta ,Q}(t)\Big).\end{equation}
 Such an $F$ exists since $\log|w-t|$ is continuous for $t\in N$ and $w\in W.$

Taking exponentials
\begin{equation}\label{l23} \prod_{j=2}^n|w-z_j|^{\beta }\geq e^{(n-1)(-\epsilon+\beta\int\log |w-t|d\mu_{K,\beta ,Q}(t))}.\end{equation}
Now  we have (see (\ref{e6}))
\begin{equation}\label{l24}\frac{1}{Z_{n-1,\beta ,\frac{nQ}{n-1}}(N)}
\int_{N^{n-1}\setminus\tilde{F}_{n-1}}|D(z_2,\ldots ,z_n)|^{\beta }e^{-2n[Q(z_2)+\ldots +Q(z_n)]}d\tau \end{equation}
$$\leq  O(e^{-cn^2}).$$
So  

 \begin{equation}\label{l25}1-O(e^{-cn^2})\leq\frac{1}{Z_{n-1,\beta ,\frac{nQ}{n-1}}(N)}\int_{\tilde{F}_{n-1}}|D(z_2,\ldots ,z_n)|^{\beta }e^{-2n[Q(z_2)+\ldots +Q(z_n)]}d\tau \end{equation}
Using the  inequalities in (\ref{l23}) and (\ref{l25}), and the fact that $\tau (W)\neq0,$  we have,
\begin{equation}\label{l112}\liminf_n1/n\log (Prob_{n,\beta ,Q}\{w\in W\})\end{equation}
$$\geq \beta \rho + inf_{w\in W}(\beta \int_K\log|w-t|d\mu_{K,\beta ,Q}-2Q(w)).$$
By (\ref{lc}) it follows that
\begin{equation}\label{l13}\inf_{\{W|z_1\in W\}}\lim_n\frac{1}{n}\log ( Prob_{n,\beta ,Q}\{w\in W\})\end{equation}
$$= \beta(\rho +  \int_K\log|z_1-t|d\mu_{K,\beta ,Q})-2Q(z_1).$$

To complete the l.d.p. we must  consider the case when $z_1\in S_R$ and to do so we must estimate $Prob_{n,\beta ,Q}\{w\in W\}$ when $W\cap S_R\not=\emptyset.$  In fact, we will show that, in this case, \begin{equation}\label{l131}\lim_n\frac{1}{n}\log (Prob_{n,\beta ,Q}\{w\in W\})=0.\end{equation}\\
Now,
$$n Prob_{n,\beta ,Q}\{w\in W\}\geq Prob_{n,\beta ,Q }\{\text{at least one of }w, z_2,\ldots ,z_n\in W\}$$  $$=1-Prob_{n,\beta ,Q }\{\text{each  of }w,z_2,\ldots ,z_n \in K\setminus W\}.$$
Since $W\cap S_R\neq\emptyset$ the support of the weighted  equilibrium measure for $R$ on $K\setminus W$   cannot be $S_R$. This implies, using  Proposition 4.1  and the minimizing property of the equilibrium measure, that $\limsup_n \sup_{K\setminus W}A_{n,\beta ,Q}(z)^{\frac{1}{n^2}}\leq \gamma -\eta,$ for some $\eta>0.$ (Recall that $\gamma= \limsup_n \sup_KA_{n,\beta ,Q}(z)^{\frac{1}{n^2}}.$) Then one can use Proposition 4.2  to obtain

\begin{equation}\label{l14} Prob_{n,\beta ,Q}\{w\in W\}\geq\frac{1}{n}(1-O(e^{-cn^2}))\end{equation}
and (\ref{l131}) follows.
\end{proof}

\section{ the unbounded case}
Let $Y$ be a closed, unbounded  subset of $\CC.$ Let $R$ be a continuous, real - valued,  superlogarithmic function on $Y.$
That is, for some $b>0,$
$$\lim_{|z|\to\infty}(R(z)-(1+b)\log|z|)=+\infty.$$
For $r>0$ we let $Y_r=:\{z\in Y||z|\leq r\}$ and we assume $Y_r$ is regular for $r$ sufficiently large. 
Also, for $r$ sufficiently large $V_{Y,R}=V_{Y_r,R} $ (\cite{ST},Appendix B, Lemma 2.2) and $S_R^*\subset Y_r.$ We will also denote the equilibrium measure $\mu_{Y,R}$ by $\mu_{Y,\beta ,Q}.$

We will extend the l.d.p. (Theorem 6.2) from the compact case to the unbounded case.
 We will need an additional hypothesis on the growth of the measure $\tau.$ 

\begin{hypothesis}

 $\tau$ is a locally finite  measure on $Y$ satisfying:
\begin{equation}\label{ua}
\text{For some } a>0 ,\text{we have} \int_Yd\tau /|z|^a<+\infty
\end{equation}
\end{hypothesis}

We note that Lebesgue measure on $\RR$ or $\CC$ satisfies Hypothesis 7.1.

The following theorem will extend Lemma 3.11 to the unbounded case. Both Lemma 3.11 and Theorem 7.2 are based on (\cite{ST}, Chapter  III,  Theorem  6.1). They shows that the $L^{\beta} $ norm of a weighted polynomial "lives" on $S_R^*.$

\begin{theorem}Let $Y$ be a closed,  unbounded subset of  $\CC$ with $Y_r$ regular for $r$ sufficiently large. Let $R$ be a superlogarithmic real-valued function on $Y.$  
 Let $\beta>0$  and let  $N\subset Y$  be a compact neighbourhood of $S_R^*$. Let $\tau$ be a measure on $Y$ such that Hypotheses 3.9 and 7.1 are satisfied.
 Then there is a constant $c>0$ (independent of $n,p$) such that:$$\int_Y|e^{-nR}p|^{\beta}d\tau\leq (1+O(e^{-nc}))\int_N|e^{-nR}p|^{\beta}d\tau$$
for all $p\in\mathcal{P}_n.$\end{theorem}
\begin{proof}
We need only consider the case $N=Y_r$ for $r$ large and then the result will follow by using  Lemma 3.11.

We first normalize the polynomial so that $||e^{-nR}p||_{S_R}=1.$ To prove the theorem it will suffice to show that
\begin{equation}\label{u2}\int_{Y\setminus Y_r}|e^{-nR(z)}p(z)|^{\beta}d\tau\leq Ce^{-nc},\end{equation} for some constants $C,c>0$ and
\begin{equation}\label{u3}\liminf_n \Big(\int_Y|e^{-nR}p|^{\beta}d\tau \Big)^{1/n}\geq 1.\end{equation} 

We will use the estimate (\cite{ST} , Appendix B, Theorem 2.6 (ii)) 
\begin{equation}\label{u4}|e^{-nR(z)}p(z)|\leq||e^{-nR(z)}p(z)||_{S_R}
\exp(n(V_{Y,R}(z)-R(z)) \end{equation}for $z\in Y.$\\

Using (\ref{p0}) we have, since $V_{Y,R}\leq\log |z| +C$ for $|z|$ large
 \begin{equation}\label{u5}V_{Y,R}(z)-R(z)\leq -\frac{b}{2}\log|z|\end{equation}
for $|z|$ large .\\
Hence, for $r$ large
\begin{equation}\label{u6}\int_{Y\setminus Y_r}|e^{-nR(z)}p(z)|^{\beta}d\tau\leq C\int_{Y\setminus Y_r}\frac{d\tau}{|z|^{\frac{nb\beta}{2}}}
\end{equation}
$$\leq Cr^{-\frac{nb\beta}{2}}r^a\int_{Y\setminus Y_r}\frac{d\tau}{|z|^a}\leq Ce^{-nc}.$$
for constants $C,c>0$ and all $n$ sufficiently large.\\

Now,\begin{equation}\label{u9}\int_Y|e^{-nR(z)}p(z)|^{\beta}d\tau\geq \int_{Y_{r}}|e^{-nR(z)}p(z)|^{\beta}d\tau\end{equation}and (\ref{u3}) follows from the proof of  Lemma 3.11.

\end{proof}
Let $Y$ be a closed,  unbounded subset of $\CC$ with $Y_r$ regular for $r$ sufficiently large. Let $R$ be a superlogarithmic real-valued function on $Y.$  We consider $\beta$ ensembles $Prob_{n,\beta ,Q}$ on $Y$. 
\begin{theorem}Let $Y$ and $R$ be as above. Assume that the measure $\tau$ satisfies Hypotheses 3.9 and 7.1 and supp($\tau )=Y.$ Also assume that Hypothesis 6.1 holds. Then the sequence of probability measures defined by, for $W$ an open subset of $Y$,  $\psi_n(W)=Prob_{n,\beta ,Q}\{z_1\in W\}$  satisfy an l.d.p.  with speed $n$ and rate function
$$\mathbb{J}_{Y,\beta ,Q}(z_1)=\beta ( R(z_1)-V_{Y,R}(z_1)) =2Q(z_1)-V_{Y,R}(z_1).$$\end{theorem}
\begin{proof}
As noted in the proof of Theorem 6.2, $$\mathbb{J}_{Y,\beta ,Q}(z_1)=2Q(z_1)-\beta (\int_Y\log |z_1-t|d\mu_{Y,\beta ,Q}(t)+\rho).$$
To  prove the l.d.p. will show that equation  (\ref{lr}) holds,  together with exponential tightness.

We consider an open set $W\subset Y.$  We may assume that $W\subset Y_r.$
Applying the results of section 6 to $\beta$ ensembles on the compact set $Y_r$ for $r$ sufficiently large, 
we have \begin{equation}\label{u14}\inf_{\{W|z_1\in W\}}\lim_n\frac{1}{n} \log Prob_{n,\beta ,Q}\{z_1\in W\}=\beta(V_{Y,R}(z_1)-R(z_1)).\end{equation}
We will show the same result holds for $\beta$  ensembles on $Y.$

Let $N\subset Y$ be a compact neighbourhood of $S_R^*.$
Now $A_{n,\beta ,Q }(z)$  and  $\int_WA_{n,\beta ,Q}(z)d\tau (z_1)$, are, in each variable of the form $e^{-2nQ}|p|^{\beta}$ for a polynomial $p\in\mathcal{P}_n$. Hence by repeated use of Theorem 7.2  we have,\\ ( see also section 6)
 
\begin{equation}\label{u12}\int_{Y^n}A_{n,\beta ,Q}(z)d\tau (z)\leq(1+O(e^{-cn}))\int_{N^n}A_{n,\beta ,Q}(z)d\tau (z),\end{equation}
and\begin{equation}\label{u121} \int_W\int_{Y^{n-1}}A_{n,\beta ,Q}(z)d\tau (z)\leq(1+O(e^{-cn}))\int_W\int_{N^{n-1}}A_{n,\beta ,Q}(z)d\tau (z).\end{equation}
Note that (\ref{u12}) shows that the integrals defining the normalizing constants (\ref{i3}) are finite.
It also follows that taking $N=Y_r$ whether we consider $\beta$ ensembles  on $Y$ or $Y_r$ that $\lim_n\frac{1}{n} \log Prob_{n,\beta ,Q}\{z_1\in W\}$ will be the same.

To complete the large deviation property in the unbounded case we must establish exponential tightness. That is :
\begin{equation}\label{u15}\lim_n\frac{1}{n}\log Prob_{n,\beta ,Q}\{|z_1|>r\}\rightarrow -\infty \text{ as } r\rightarrow\infty\end{equation}

Proceeding as in the proof of  Theorem 7.2  but without normalizing the polynomial, we obtain
\begin{equation}\label{u151}\int_{Y\setminus Y_r}|e^{-nR(z)}p(z)|^{\beta}d\tau\leq  ||e^{-nR(z)}p(z)||_{S_R}^{\beta}\int_{Y\setminus Y_r}\frac{d\tau}{|z|^{\frac{nb\beta}{2}}}\end{equation}
$$\leq C  ||e^{-nR(z)}p(z)||_{S_R}^{\beta}r^{\frac{-nb\beta}{2}+a}$$
where $C$ is independent of $n,p.$\\

By the weighted BM inequality, we have
\begin{equation}\label{u152} ||e^{-nR(z)}p(z)||_{S_R}^{\beta}\leq C(1+\epsilon)^n\int_{Y_r}|e^{-nR}p|^{\beta}d\tau\end{equation}
$$\leq C(1+\epsilon)^n\int_Y|e^{-nR}p|^{\beta}d\tau.$$
Now$$Prob_{n,\beta ,Q}\{|z_1|>r\}=\frac{1}{Z_{n\beta ,Q}}\int_{Y\setminus Y_r}\int_{Y^{n-1}}A_{n,\beta, Q}(z)d\tau(z)$$ and using (\ref{u151}) and (\ref{u152}) we have
$$Prob_{n,\beta ,Q}\{|z_1|>r\}\leq C(1+\epsilon )^nr^{\frac{-nb\beta}{2}+a}\frac{1}{Z_{n\beta ,Q}}\int_{Y^n}A_{n,\beta, Q}(z)d\tau(z).$$\\
But$\frac{1}{Z_{n,\beta ,Q}}\int_{Y^n}A_{n,\beta, Q}(z)d\tau(z)=1$
and since $\epsilon > 0$ is arbitrary, Theorem 7.3  is established.

\end{proof}
\begin{remark}

 Consider the probability distributions on $Y,$ given by, for $W$ an open subset of $Y:$  $$\psi'_n(W):=Prob_{n,\beta ,Q}\{\text{at least one of the coordinates } z_1,z_2,\ldots , z_n\in W\}.$$ Then the sequence of probability distributions $\psi_n'$ under the  hypothesis of Theorem 7.3  satisfy the same l.d.p. as $\psi_n.$ This is because the joint probability distribution is symmetric in $z_1,\ldots,z_n$ so
$$nProb\{z_1\in W\}\geq Prob\{\text{at least one of }z_1,\ldots,z_n\in W\}\geq Prob\{z_1\in W\},$$ and  the l.d.p for $\psi_n$  has speed $n.$
\end{remark}

\end{document}